\documentclass[10pt]{article}
\usepackage{amssymb,amsmath,amsthm,latexsym}
\usepackage{amssymb,graphicx}
\newtheorem{theorem}{Theorem}

\newtheorem{conjecture}[theorem]{Conjecture}
\newtheorem{corollary}[theorem]{Corollary}

\newtheorem{definition}[theorem]{Definition}

\newtheorem{lemma}[theorem]{Lemma}

\newtheorem{question}[theorem]{Question}
\newtheorem{proposition}[theorem]{Proposition}

\def\slr{\widetilde{SL(2,\mathbb{R})}}
\def\hr{\mathbb{H}^2 \times \mathbb{R}}

\title{Asymptotically CAT(0) Groups}
\author{Aditi Kar}
\date{November 9,2009}

\begin{document}
\maketitle
\begin{abstract} \noindent We develop a general theory for asymptotically CAT(0) groups; these are groups acting geometrically on a geodesic space, all of whose asymptotic cones are CAT(0). \end{abstract}

\section{Introduction}

Metric spaces of non-positive curvature have been the central objects of study among geometric group theorists for more than two decades. A fundamental theorem of Riemannian Geometry states that the universal cover of a complete Riemannian $n$-manifold with constant sectional curvature is isometric to $\mathbb{H}^n$, $\mathbb{R}^n$ or $\mathbb{S}^n$. In geometric group theory, $CAT(\kappa)$ spaces are `modelled' on these three spaces. They encapsulate in a metric fashion, the traditional notion of sectional curvature that is bounded above. 

The study of groups of isometries of non-positively curved spaces has proved to be very fruitful in enhancing our understanding of finitely presented groups. A CAT(0) space (see \cite{BH}) has many fascinating properties. Indeed, it is always contractible and it exhibits a rather desirable local-to-global phenomenon.

A group $G$ is said to act geometrically on a metric space $(X,d)$ if it acts properly and co-compactly by isometries on $X$.  Groups acting geometrically on CAT(0) spaces or CAT(0) groups have soluble word and conjugacy problems. All free abelian subgroups of CAT(0) groups are finitely generated and a CAT(0) group can contain only finitely many conjugacy classes of finite subgroups. These are a few highlights of the subject but they portray how well the geometry of the space complements the algebra of the group acting geometrically on it. 

The other profound notion that greatly enhanced our knowledge of infinite groups was $\delta$-hyperbolicity (found in \cite{G}). The isoperimetric function of a hyperbolic group is linear and so there is an efficient solution to the word problem. Hyperbolic groups satisfy all the desirable properties we mentioned earlier with regard to CAT(0) groups. Further, many notoriously difficult conjectures like the Novikov and the Baum-Connes are known to be true for this class of groups. 

There seems to be great merit in generalizing the notion of non-positive curvature for the purpose of studying infinite discrete groups. It has been done in the past. Gromov suggested the idea of relatively hyperbolic groups which turned out to be invaluable in the study of fundamental groups of complex hyperbolic manifolds with cusps. More recently, systolic complexes were introduced by Tadeusz Januszkiewicz and Jacek Swiatkowski in \cite{JS}. These are simply connected simplicial complexes satisfying a local combinatorial condition reminiscent of nonpositive curvature. Systolic groups too share many properties with $CAT(0)$ groups. In solving Novikov's Conjecture for hyperbolic groups, Kasparov and Skandalis introduced the class of bolic groups (see \cite{bolic}). Gromov recently introduced the notion of $CAT_{\delta}(\kappa)$ spaces in \cite{G1} and later with Delzant studied the $CAT_{\delta}(-1)$ spaces in \cite{delzant}. 

In this article we introduce the notion of a metric space and a group being asymptotically CAT(0). 

\begin{definition} A metric space $X$ is said to be asymptotically CAT(0) if all asymptotic cones of $X$ are CAT(0). \end{definition}

Heuristically, asymptotic cones provide the perspective of a metric space from infinitely far away. Hence, an asymptotically CAT(0) space appears to have non-positive curvature when viewed from increasingly distant observation points. In this sense, the class of asymptotically CAT(0) spaces is a generalizaion of the metric spaces of non-positive curvature.

Few metric spaces have unique asymptotic cones. The isometry type of an asymptotic cone depends heavily on the choice of the ultrafilter and the base point. However, if the metric space supports a co-compact group action, then one can remove this dependence on the base point. 

The dependence on the choice of the ultrafilter is a far more delicate matter. In \cite{TV}, Simon Thomas and Boban Velickov present an example of a finitely generated group whose Cayley graph possesses non-isometric asymptotic cones for different choices of non-principal ultrafilters.

A weaker notion to `asymptotically CAT(0)' is that of a metric space being lacunary CAT(0), i.e. it has at least one CAT(0) asymptotic cone. In \cite{OOS}, the authors provide an example of a finitely generated group $G$ such that at least one asymptotic cone of $G$ is the infinitely branching homogeneous $\mathbb{R}$-tree while some other asymptotic cones of $G$ are not even simply connected. Thus there exist metric spaces which are lacunary CAT(0) but not asymptotically CAT(0). 

\begin{definition} A group $G$ is asymptotically CAT(0) if it acts properly and co-compactly by isometries (geometrically, for short) on an asymptotically CAT(0) geodesic space.\end{definition}

The principal motive of this article is to present examples and study properties of asymptotically CAT(0) groups. The class of asymptotically CAT(0) groups is a natural enlargement of the collection of non-positively curved groups. However, as far as the author knows, there has been no previous exploration of this concept. 

First, a metric characterization of asymptotically CAT(0) geodesic spaces is provided; this result best describes the local-to-global behaviour of asymptotically CAT(0) spaces.

\medskip

\noindent \textbf{Theorem A} \textit{(Theorem \ref{metric}) A geodesic space is asymptotically CAT(0) if and only if balls of radius $r$ are coarsely CAT(0) with an error of $f(r)$, where the function $f: \mathbb{R}_+ \rightarrow \mathbb{R}_+$ is sublinear.} 

\medskip

We give an example of a geodesic space (Proposition \ref{wrinkledplane}, `The plane with the wrinkled quadrant') in which the function $f$ from Theorem A is not constant (as is the case with hyperbolic or CAT(0) groups). 

Asymptotically CAT(0) groups have nice finiteness properties: indeed, any such group is of type $F_\infty$. The Dehn function is subcubic, as a consequence of which the word problem for an asymptotically CAT(0) group is solvable. Finally, every quasi-isometrically embedded nilpotent subgroup of an asymptotically CAT(0) group is virtually abelian. These are explained in Section \ref{prop}. Using the metric characterization from Theorem A, we investigate finite subgroups of asymptotically CAT(0) groups. 

\medskip

\noindent \textbf{Theorem B}. \textit{(Theorem \ref{finitesubgroup}) An asymptotically CAT(0) group has finitely many conjugacy classes of finite subgroups. }

\medskip

In Section \ref{combination}, techniques are provided for combining asymptotically CAT(0) groups. All of these results are not surprisingly much like their CAT(0) counterparts. The class of asymptotically CAT(0) groups is clearly closed under taking finite direct products. One can form amalgams and HNN extensions of asymptotically CAT(0) groups, provided the hypotheses of \textit{Theorems \ref{amalgam}, \ref{hnn}} are true. In particular the class of asymptotically CAT(0) groups is shown to be closed under free products with amalgamation and HNN extensions along finite subgroups and prove the theorem below. 

\medskip

\noindent \textbf{Theorem C} \textit{(Theorem \ref{amalgamfinite})The class of asymptotically CAT(0) groups is closed under taking free products with amalgamation and HNN extensions along finite subgroups.}

\medskip

In the final sections, we concentrate on providing examples of asymptotically CAT(0) groups. Hyperbolic groups and CAT(0) groups are known to be asymptotically CAT(0). Indeed, the asymptotic cone of every hyperbolic group is an $\mathbb{R}$-tree (See \cite{D}). On the other hand, every asymptotic cone of a CAT(0) space is also CAT(0). A proof of this may be found in \cite{BH}. The $CAT_{\delta}(-1)$ spaces of \cite{delzant} are hyperbolic and thus groups acting geometrically on $CAT_{\delta}(-1)$ spaces are asymptotically CAT(0). Also, bolic spaces in the sense of \cite{bolic} are asymptotically CAT(0); a discussion on this may be found in \cite{thesis}.

A potentially rich source of examples is the class of relatively hyperbolic groups. 

\medskip 

\noindent \textbf{Theorem D.} \textit{(Theorem \ref{RelH}) Suppose that a group $G$ is relatively hyperbolic with respect to a subgroup $H$. If $H$ is asymptotically $CAT(0)$, then $G$ is asymptotically CAT(0). }

\medskip

It follows from the above theorem and Elsner's main result from \cite{TE} that systolic groups with isolated flats are asymptotically CAT(0). (See Corollary \ref{systolic}). It seems reasonable to conjecture that all systolic groups are asymptotically CAT(0). 

In many ways, $\slr$ is the most intriguing of Thurston's eight geometries. Co-compact lattices in $\slr$ are neither hyperbolic, nor are they $CAT(0)$. Now, it is well known that $\slr$ is quasi-isometric to $\hr$ and so every asymptotic cone of $\slr$ is homeomorphic to the direct product of an $\mathbb{R}$-tree with the real line. However, this fact alone does not determine the structure of the asymptotic cones up to isometry, unless the multiplicative constant for the quasi-isometry is 1. No demonstration of the quasi-isometry known to the author gives such a precise constant.

We exploit the Riemannian geometry of $\slr$ (endowed with the Sasaki metric) to show that there is a $(1,\pi)$ quasi-isometry between $\slr$ and the CAT(0) space, $\mathbb{H}^2\times \mathbb{R}$. The induced maps at the level of asymptotic cones are therefore, isometries. Moreover, `taking asymptotic cones' commutes with direct products. Hence, 

\medskip

\noindent \textbf{Theorem E.} \textit{(Theorem \ref{slr}) Co-compact lattices in $\slr$ are asymptotically CAT(0). More precisely, every asymptotic cone of $\slr$, endowed with the Sasaki metric is isometric to a direct product (with the square metric) of the infinitely branching homogeneous $\mathbb{R}$-tree and the real line.}

\medskip 

\noindent \textbf{Acknowledgements.} This work is part of my doctoral dissertation presented at The Ohio State University. I am indebted to my superviser, Indira Chatterji for her guidance and advice. I wish to thank T. Delzant,  J-F. Lafont, T. Riley, and D. Groves for the many informative discussions I had with them in the course of writing this paper. 

\section{Asymptotic Cones and Quasi-isometries}\label{intro}

There are many excellent expositions on asymptotic cones and I will not attempt to define them here. The reader may find a short introduction in \cite{BH}.

\noindent \textbf{Notation} The asymptotic cone of a metric space $(X,d)$ with respect to a non-principal ultrafilter $\omega$, scaling sequence $(a_n)$ and basepoints $(p_n)$ is denoted as $Cone_{\omega}(X,(a_n),(p_n))$ or as $Cone_{\omega}(X)$ or simply as $X_{\omega}$.

\noindent \textbf{The canonical asymptotic cone.} If a metric space $X$ supports a co-compact group action then the isometry types of its asymptotic cones do not depend on the base point. Let $x \in X$. Then for all sequences $(x_n)$ of points of $X$, $Cone_{\omega}(X,(x_n),(a_n))$ is isometric to $Cone_{\omega}(X,(x),(a_n))$. So we will refer to the asymptotic cone of a space $X$ with respect to the `constant' sequence, as the \emph{canonical asymptotic cone} of $X$ for the given choice of sequence $(a_n)$ and $\omega$. 

\begin{definition}\label{quasiisometry} Let $f:X \rightarrow Y$ be a map of metric spaces. If there exist constants $\lambda \geq 1$ and $\epsilon \geq 0$ such that \[\frac{1}{\lambda} d(x,x') - \epsilon \leq d(f(x),f(x')) \leq \lambda d(x,x') + \epsilon, \textrm{ for all } x, x' \in X,\] then $f$ is called a $(\lambda,\epsilon)$-quasi-isometric embedding of $X$ into $Y$. If in addition, there exists a constant $c\geq 0$ such that $d(y,f(X))\leq c$ for all $y\in Y$, then $f$ is called a quasi-isometry. \end{definition}

\begin{lemma}\label{qi} A quasi-isometry induces a bilipschitz homeomorphism at the level of asymptotic cones. \end{lemma}
\begin{proof} Let $f:X \rightarrow Y$ be a quasi-isometry between metric spaces, with associated constants $\lambda$, $\epsilon$ and $c$, as above. Let $\omega$ be a non-principal ultrafilter and $(a_n)$ a sequence of positive real numbers such that $\lim_{n\rightarrow \infty}a_n=\infty$. Let $(p_n)$ be a sequence of points from $X$. 

Set $X_\omega = Cone_\omega(X,(a_n),(p_n))$ and $Y_\omega= Cone_\omega(Y, (a_n),(f(p_n)))$. The function $f$ induces a map $F: X_\omega \rightarrow Y_\omega$, defined by $\lim_\omega x_n \mapsto \lim_\omega f(x_n)$. Observe that if $\lim_\omega x_n$ denotes an equivalence class in $X_\omega$, then $d(x_n,p_n)/a_n$ is a bounded sequence of real numbers. Since $f$ is a quasi-isometry, it follows that $d(f(x_n),f(p_n))/a_n $ is also bounded and so the map $F$ is well-defined. 

Now let $\lim_\omega x_n$ and $\lim_\omega x'_n$ be elements of $X_\omega$. Then, for every $n \in \mathbb{N}$, we have \[\frac{1}{\lambda} \frac{d(x_n,x'_n)}{a_n} - \frac{\epsilon}{a_n} \leq \frac{d(f(x_n),f(x'_n)) }{a_n}\leq \lambda \frac{d(x_n,x'_n)}{a_n} + \frac{\epsilon}{a_n}. \textrm{ This implies that}\]
\[\frac{1}{\lambda} d_\omega(\lim_\omega x_n, \lim_\omega x'_n)\leq d_\omega(F(\lim_\omega x_n),F(\lim_\omega x'_n)) \leq \lambda d_\omega(\lim_\omega x_n,\lim_\omega x'_n).\quad \quad\] 

The function $f$ has a `quasi-inverse', $g: Y \rightarrow X$, which is a $(\lambda', \epsilon')$ quasi-isometry. Moreover, there exists a constant $k \geq 0$ such that $d(gf(x),x)\leq k$ and $d(fg(y),y)\leq k$, for all $x\in X$ and for all $y \in Y$. The function $g$ induces a map $G: Y_\omega \rightarrow X_\omega$ at the level of asymptotic cones. As before, every pair $(\lim_\omega y_n,\lim_\omega y'_n)$  of points from $Y_\omega$ satisfies \[\frac{1}{\lambda'} d_\omega(\lim_\omega y_n, \lim_\omega y'_n)\leq d_\omega(G(\lim_\omega y_n),G(\lim_\omega y'_n)) \leq \lambda' d_\omega(\lim_\omega y_n,\lim_\omega y'_n).\] 
Moreover, $GF(\underline{x})=\underline{x}$ for all $\underline{x}\in X_\omega$ and $FG(\underline{y})=\underline{y}$ for all $\underline{y}\in Y_\omega$. 

We conclude from the above discussion that $F$ is a bilipschitz homeomorphism between $X_\omega$ and $Y_\omega$. \end{proof}

\begin{corollary}\label{isometry} A $(1,\epsilon)$ quasi-isometry induces an isometry at the level of asymptotic cones. \end{corollary}

\section{Asymptotically CAT(0) metric spaces}\label{character}

The purpose of this section is to obtain a characterization of asymptotically CAT(0) spaces in terms of their metric properties. First, we need to introduce a coarse version of the CAT(0) inequality.  

A geodesic segment, denoted $[xy]$, joining two points $x$ and $y$ of a metric space $(X,d)$ is the isometric image of a path of length $d(x,y)$ joining $x$ and $y$. A geodesic triangle in $X$ consists of its three vertices, call them $x$, $y$, $z $ and a choice of geodesic segments $[xy]$, $[yz]$ and $[zx]$ joining these vertices. We will denote such a geodesic triangle by $\triangle(x,y,z)$.

A triangle $\bar{\triangle }(\bar{x},\bar{y},\bar{z})$ in $\mathbb{E}^2$ is called a comparison triangle for $\triangle(x,y,z)$ if $d(x,y)= d(\bar{x},\bar{y})$, $d(y,z)= d(\bar{y},\bar{z})$, and $d(z,x)= d(\bar{z},\bar{x})$. It is a consequence of the triangle inequality that given a triangle in $X$, there is always a comparison triangle in $\mathbb{E}^2$.

\begin{definition}\label{maindef} Let $\triangle$ be a geodesic triangle in $X$ with comparison triangle $\bar{\triangle}$ in $\mathbb{E}^2$. Let $\delta > 0$. Then, $\triangle$ is said to satisfy the $\delta$-CAT(0) inequality if for all $p$, $q \in \triangle$ and comparison points $\bar{p}$, $\bar{q}$, we have 
\[d(p,q) \leq d(\bar{p},\bar{q}) + \delta.\]
$X$ is called a $\delta$-CAT(0) space if $X$ is a geodesic metric space  and there is a $\delta \geq 0$ such that all geodesic triangles in $X$ satisfy the $\delta$-CAT(0) inequality. \end{definition}

The $\delta$-CAT(0) inequality for triangles extends easily to geodesic quadrilaterals. This is the content of the following lemma and the proof of the lemma is left to the reader. 

\begin{lemma} \label{four} Let $X$ be a $\delta$-CAT(0) space and let $(x_1,x_2,x_3,x_4)$ be a geodesic quadrilateral in $X$. Then there exists a convex quadrilateral in the Euclidean plane with vertices $(\bar{x}_1,\bar{x}_2,\bar{x}_3,\bar{x}_4)$, such that $d(x_i,x_{i+1})$ $= d(\bar{x}_{i},\bar{x}_{i+1})$ for all $i$, modulo $4$, and $d(x_i,x_j)$ $\leq d(\bar{x}_{i}, \bar{x}_{j}) + \delta$, for all $i \neq j=1,2,3,4$.\end{lemma}

We are now in a position to state our characterization of asymptotically CAT(0) geodesic spaces.  

\begin{theorem}\label{metric} A geodesic metric space is asymptotically CAT(0) if and only if there exists a function $f: \mathbb{R}_+\rightarrow \mathbb{R}_+$ such that $\lim_{r\rightarrow \infty} \frac{f(r)}{r}=0$ and every ball of radius $r$ in $X$ is $f(r)$-CAT(0).\end{theorem}

\textbf{Caveat.} In the statement above we do not assume that the balls in $X$ are convex. We simply mean that any geodesic triangle in $X$ with vertices in a ball of radius $r$ satisfies the $f(r)$-CAT(0) inequality. 

\begin{proof}[Proof of Theorem \ref{metric}] We first consider the sufficiency statement. Take a 4-tuple of points $(x_1, x_2, x_3, x_4) \in Cone_{\omega}(X) $, where $ x_i = (x_{i,n}) $ for $ i = 1, 2, 3, 4 $. Note that for each $ n $, $(x_{1,n},x_{2,n},x_{3,n},x_{4,n}) $ is a 4-tuple of points in $(X,d_n)$ of some diameter $r_n$. Since any ball of radius $r_n$ satisfies a $f(r_n)$-CAT(0) inequality for triangles, we know that it also satisfies a $f(r_n)$-CAT(0) inequality on quadrilaterals. This implies that there is a 4-tuple of points $(y_{1,n},y_{2,n},y_{3,n},y_{4,n}) $ in the Euclidean plane $\mathbb{E}^2$ such that $d(x_{i,n},x_{i+1,n})=d(y_{i,n},y_{i+1,n})$ for $i=1,2,3,4$, modulo 4 and $d(x_{i,n}, x_{j,n}) \leq d(y_{i,n},y_{j,n}) + 2f(r_n)$ for $1\leq i < j \leq 4$. For each $n$, we may choose $y_{1,n}$ to be the origin.

Then $ y_i = ({y_{i,n}}) $, $ i=1,2,3,4 $ is a 4-tuple of points in $Cone_{\omega}(\mathbb{E}^2)$. We know that the Euclidean plane is isometric to any of its asymptotic cones. The above construction therefore provides us with a 4-tuple of points in $\mathbb{E}^2$ which satisfies $d(x_i,x_{i+1})=d(y_i, y_{i+1})$, for $i=1,2,3,4$, modulo 4, and for $1\leq i < j \leq 4$, $d(x_i,x_j)=\lim_{\omega}d_n(x_{i,n}, x_{j,n}) \leq \lim_{\omega}(d_n(y_{i,n}, y_{j,n}) + 2\frac{f(r_n)}{a_n}) $. But, $\lim_{\omega}\frac{f(r_n)}{a_n}$=
$\lim_{\omega}\frac{f(r_n)}{r_n} \frac{r_n}{a_n}$. By hypothesis,
$\lim_{r \rightarrow \infty} \frac{f(r)}{r} =0$ and further, $\lim_{\omega}\frac{r_n}{a_n}$
is the diameter of the four tuple of points in the asymptotic cone. We therefore
conclude that $d(x_i,x_j) \leq d(y_i,y_j)$ for $1\leq i < j \leq 4$.

Conversely, suppose that all asymptotic cones of a geodesic space $X$ are CAT(0). Define $f(r)$ to be the supremum of the difference between $d(p,q)$ and $d(\bar{p},\bar{q})$, where $p$ and $q$ are points on a geodesic triangle in $X$, whose vertices lie in a ball of radius $r$. We claim that $\lim_{r \rightarrow \infty}\frac{f(r)}{r}=0$. 

Suppose not. Then there exists a non-principal ultrafilter $\omega$ and a sequence $(a_n)$ of positive real numbers such that $\lim_{n \rightarrow \infty}a_n = \infty$ and \[ \textrm{ for some } \epsilon >0,\ \omega( \{n : \frac{f(a_n)}{a_n} >2\epsilon \})=1.\] 
For each $n \in \mathbb{N}$, there exists a geodesic triangle $\Delta_n$, with vertices in a ball of radius $a_n$ in $X$ and a comparison triangle $\bar{\Delta}_n$ for $\Delta_n$ in the Euclidean plane such that a pair $p_n$, $q_n$ of points in $\Delta_n$ satisfies $f(a_n) $ $\geq d(p_n,q_n)- d(\bar{p}_n,\bar{q}_n) $ $\geq f(a_n)-1$. Here, $\bar{p}_n$ and $\bar{q}_n$ are as usual the comparison points for $p_n$ and $q_n$ in $\bar{\Delta}_n$. 

Now consider the asymptotic cone $X_\omega$ of $X$ with respect to the scaling sequence $(a_n)$, ultrafilter, $\omega$ and sequence of base points $(p_n)$. The $\omega$-limit of the triangles $\Delta_n$ is a geodesic triangle $\Delta$ in a ball of radius 1 in $X_\omega$. Observe that if $\bar{\Delta}$ denotes the $\omega$-limit of the triangles $\bar{\Delta}_n$, then $\bar{\Delta}$ is a comparison triangle for $\Delta$ in the Euclidean plane. The comparison points for $\lim_\omega p_n$ and $\lim_\omega q_n$ in $\bar{\Delta}$ are precisely the $\omega$-limits of the sequences $(\bar{p}_n)$ and $(\bar{q}_n)$, respectively.

As $\omega\{n \in \mathbb{N}\ |\ \frac{1}{a_n}< \epsilon\}=1$, we deduce that \[\omega\left\{n\in \mathbb{N}\ |\  \frac{d(p_n,q_n)}{a_n}-\frac{d(\bar{p}_n,\bar{q}_n)}{a_n}\geq \epsilon\right\}=1.\]

Therefore, $d_\omega(\lim_\omega p_n, \lim_\omega q_n)\geq d_\omega(\lim_\omega \bar{p}_n, \lim_\omega \bar{q}_n))+\epsilon$. This contradicts the assumption that $X_\omega$ is CAT(0). We conclude that $\lim_{r\rightarrow \infty}\frac{f(r)}{r}=0$. \end{proof}

Clearly, $\delta$-CAT(0) spaces are asymptotically CAT(0). One may ask if the converse is true. We now present an example of a geodesic space which is asymptotically CAT(0) but not $\delta$-CAT(0) for any $\delta \geq 0$. 

\begin{proposition}[The plane with the wrinkled quadrant] \label{wrinkledplane} There is a metric space $Y$ which is not $\delta$-CAT(0), for any $\delta\geq 0$ but all its asymptotic cones are isometric to a CAT(0) space.\end{proposition}

\begin{proof} For each integer $n \geq 2$, take $S_n$ to be the trapezium in the first quadrant of the Euclidean plane, bounded by the $x$-axis, the $y$-axis, and the lines, $x+y=n(n+1)$ and $x+y=n(n-1)$. Now let $P_n$ be a solid with five faces: the base of $P_n$ is the trapezium $S_n$; two isosceles triangles, each of base length $2n$ and side length $\sqrt{n^2 +1}$ form two of the faces. The remaining two faces are trapezia, one with sides $\sqrt{n^2+1}$, $\sqrt{2}n(n+1)$, $\sqrt{n^2+1}$ and $\sqrt{2}n^2$ and the other with sides, $\sqrt{n^2+1}$, $\sqrt{2}n(n-1)$, $\sqrt{n^2+1}$ and $\sqrt{2}n^2$. Note that each prism is of height 1. Attach $P_n$ isometrically along its base to the trapezium $S_n$. Finally, remove the interior of each $P_n$, along with the interior of the base, $S_n$. Give the resulting space $Y$, the induced path metric; it can be loosely described as `the plane with the wrinkled quadrant'. 

We claim that the space $Y$ is not $\delta$-CAT(0) for any $\delta \geq 0$ but all of its asymptotic cones are CAT(0). For each $n \in \mathbb{N}$, let $w_n$ be the point in $Y$ whose original Euclidean coordinates were $\left(\frac{n(n+1)}{2},\frac{n(n+1)}{2}\right)$. Then, the geodesic $\gamma_n$ in $Y$ joining the `origin' to $w_n$ is of length $d_n=\sqrt{2}$ $+\sum_{k=2}^{n}2\sqrt{1+\frac{k^2}{2}}$. 

Consider triangles $T_n$, one for each integer $n\geq 1$, whose vertices are at the origin, and the points $(0,n(n+1))$ and $(n(n+1),0)$. Then, $(\frac{n(n+1)}{2}, \frac{n(n+1)}{2})$ is the mid-point $m_n$ of the side $[(0,n(n+1)),(n(n+1),0)]$ and the Euclidean distance $\bar{d}_n:= d_{\mathbb{E}^2} ((0,0),m_n)$ is exactly $\sum_{k=1}^{n}k\sqrt{2}$. 

The triangle $T_n$, for each $n \geq 1$ therefore coincides along its boundary to its Euclidean comparison triangle. The difference between $d_n$ and $\bar{d}_n$ is given by 
\[ \sum_{k=2}^n\ \left(2\sqrt{\frac{k^2}{2}+1} - \sqrt{2}k\right).\]
The summand is equal to $\frac{4}{\sqrt{2k^2+4}+k\sqrt{2}}$. As this is no smaller than $\frac{4}{k(\sqrt{2}+\sqrt{3})}$ and the harmonic series diverges, $Y$ is not $\delta$-CAT(0), for any $\delta \geq 0$. 

We claim that every asymptotic cone of $Y$ is isometric to the Euclidean plane. Imagine a juxtaposition of $Y$ and $\mathbb{E}^2$ in which $Y$ lies above $\mathbb{E}^2$ and the $x$ and $y$ axes in $\mathbb{E}^2$ coincide with the copy of the axes in $Y$. There is a projection $\pi$ of $Y$ onto $\mathbb{E}^2$ that maps every point in $Y$ to the point in $\mathbb{E}^2$ directly below it. 

We want to estimate the quantity $f(p,q):=|d(p,q)-d(\pi(p),\pi(q))|$. If a path joining two points $p$ and $q$ in $Y$ crosses $n$ wrinkles, then by the triangle inequality, the difference $f(p,q)$ is at most $2n$. On the other hand, any path crossing $n$ wrinkles must travel a distance of at least $k+(k+1)\dots + (n+k)$. Now, $k+(k+1)\dots + (n+k)$ is equal to $(n+k)(n+k+1)/2-k(k+1)/2$, which is no smaller than $n^2/2$. Therefore we see that the quantity $f(p,q)$ is bounded above by a linear function of $\sqrt{d(p,q)}$. 

We deduce from the above discussion that every asymptotic cone of $Y$ is isometric to the Euclidean plane. \end{proof}

\section{Properties of Asymptotically CAT(0) groups}\label{prop}

We first use the metric characterization from the proceeding section to study the finite subgroups of asymptotically CAT(0) groups. Afterwards, we consider the word problem, finiteness properties and nilpotent subgroups of asymptotically CAT(0) groups.

\subsubsection*{Finite Subgroups}

\begin{theorem} \label{finitesubgroup} An asymptotically CAT(0) group $G$ has finitely many conjugacy classes of finite subgroups. \end{theorem}

\begin{proof} Let $Y$ be a non-empty bounded subset of a proper asymptotically CAT(0) metric space. Define

\medskip

$r_Y=\inf\{r>0 | Y \subset B(x,r)$ for some $x \in X\}$, and

\medskip

$C(Y)= \{x \in X | Y \subset B(x,r_Y)\}$.

\medskip

With this notation, $r_Y$ is the (circum)radius of $Y$ and $C(Y)$ is the set of barycentres of $Y$. By a standard argument, $C(Y)$ is not empty. We wish to estimate the diameter of $C(Y)$. Now, by Theorem \ref{metric}, we know there exists a function $f$ and $a>0$ such that if $r>a$, then $f(r) < \frac{r}{32}$.

Choose $x_1$, $x_2$ $\in C(Y)$ and let $\epsilon > 0$ be given. For each $y \in Y$, consider a geodesic triangle $[y, x_1, x_2]$ along with $[O_y, \bar{x}_1, \bar{x}_2]$, its comparison triangle in the Euclidean plane. Suppose $m$ is the midpoint of the geodesic joining $x_1$ and $x_2$. Denote its comparison point in $[O_y, \bar{x}_1, \bar{x}_2]$ by $m_y$. Now, if $d(m_y, O_y) \leq r_Y - f(r_Y) - \epsilon $ for all $y \in Y$, then $d(m,y) \leq r_Y - \epsilon$ for all $y \in Y$ and this violates the definition of $r_Y$. Therefore there must exist some $z \in Y$ for which the distance between the points $O_z$ and $m_z$ exceeds $r_Y -(f(r_Y) + \epsilon)$. Thus, \\
$d(\bar{x}_1,\bar{x}_2)\leq 2\sqrt{d(\bar{x}_1,O_z)^2-d(O_z,m_z)^2} $ $\leq 2\sqrt{r_Y^2- (r_Y -(f(r_Y) + \epsilon))^2}$\\
$\Rightarrow$ $2r_{C(Y)} \leq 2\sqrt{2r_Yf(r_Y)  - f(r_Y) ^2 }$. Hence, if $r_Y > a$, then $r_{C(Y)} \leq \frac{r_Y}{4}$.

Now let $G$ be an asymptotically CAT(0) group with associated space $X$. Note that there is no need to assume that the metric space $X$ is complete. As $G$ acts geometrically on $X$, by Exercise 1 in Section I.8.4 of \cite{BH}, $X$ is a proper geodesic space. This implies that centres exist for bounded subsets of $X$. 

Let $H$ be a finite subgroup of $G$. Fix $x \in X$. Set $Y=Hx$, the orbit of $x$ under the action of $H$. Then $Y$ is a bounded subset of $X$. If $r_Y > a$, then we inductively define a sequence $(Y_n)$ of subsets of $X$, by setting $Y_0$ to be $Y$ and $Y_n$, to be $C(Y_{n-1})$. We deduce from the previous paragraphs that \[\textrm{for\ } n > \log_4{\frac{r_Y}{a}}\ ,\ \textrm{we\ have\ } r_{Y_n}<a.\]

Choose $m$ to be the least such $n$. Note that, by construction, the sets $Y_n$ are invariant under the action of $H$. As the action is geometric, there exists a ball $B(x,D)$ in $X$ such that $G.B(x,D)=X$. Let $\bar{x}$ be an element of $Y_m$. There exists some $g \in G$ with $d(gx,\bar{x})\leq D$. Hence, for any $z \in Y_m$, we have 

$d(g^{-1}z,x) \leq d(g^{-1}z, g^{-1}\bar{x}) + d(g^{-1}\bar{x}, x) \leq 2a + D.$

Let $h \in H$. Then,

$d(g^{-1}hgx,x) $ $\leq d(g^{-1}hgx, g^{-1}hg.g^{-1}z)$ $ + d(g^{-1}hg.g^{-1}z, x)$

$= d(x,g^{-1}z) + d(g^{-1}(hz), x)$.

\medskip

The set $g^{-1}Y_m$ is invariant under the action of $g^{-1}Hg$ and therefore, $g^{-1}(hz)$ $ \in g^{-1}Y_m$. Hence, $d(g^{-1}hgx,x)$ $\leq 2(2a + D) $ and $g^{-1}Hg.x $ $\subset B(x,2(2a + D))$.

The properness of the action of $G$ ensures that there are only finitely many subgroups with the property that the orbit of a point $x$ lies in the $2(2a + D)$-ball around $x$. This proves the theorem. \end{proof}

\subsubsection*{Other Finiteness Properties}\label{classifyingspace}

Let $G$ be an asymptotically CAT(0) group. There exists a space $X$ on which $G$ acts geometrically and such that all asymptotic cones of $X$ are CAT(0). Any CAT(0) space is contractible. The fact that all asymptotic cones of $X$ are contractible has implications for the finiteness properties of $G$. In fact, it implies that $G$ is of type $F_\infty$. 

Given a group $G$, a $K(G,1)$ is a path-connected space whose fundamental group is isomorphic to $G$ and which has a contractible universal covering space. The existence of a `nice' $K(G,1)$ has a central place in algebraic topology. 

\begin{definition} A group $G$ is said to be of type $F_n$ if there exists a CW-complex $K(G,1)$, whose $n$-skeleton is finite.\end{definition}

\noindent \textbf{Remarks} We say that $G$ is of type $F_\infty$ if there is a CW-complex $K(G,1)$ with finitely many cells in each dimension.

\begin{theorem}(Theorem 2.6.D of \cite{TR}) If $G$ is a finitely generated group with a word metric such that all asymptotic cones of $G$ are $n$-connected, then $G$ is of type $F_{n+1}$.\end{theorem}

Therefore, if all asymptotic cones of $G$ are contractible, then there exists a CW-complex with finitely many cells in each dimension and whose fundamental group is isomorphic to $G$. 

Now, suppose that $G$ is asymptotically CAT(0) and $X$ is as above. The group $G$ is finitely generated and by the $Sv\breve{a}rc$-$Milnor$ $Lemma$ (Proposition I.8.19 in \cite{BH}), $G$ with any word metric, is quasi-isometric to the space $X$. But, by Lemma \ref{qi}, a quasi-isometry induces a bilipschitz homeomorphism at the level of asymptotic cones. Since $X$ is asymptotically CAT(0), all asymptotic cones of $X$ are contractible. This implies that all asymptotic cones of $G$ are also contractible and hence, $G$ is of type $F_{\infty}$. We have therefore the following proposition. 
\begin{proposition}An asymptotically CAT(0) group is of type $F_\infty$.\end{proposition}

\subsubsection*{The Word Problem}\label{wordproblem}
\begin{theorem} (Theorem 4.6 in \cite{D}) Let $X$ be a geodesic space. If the isoperimetric function for every asymptotic cone of $X$ is quadratic then the following is true: for every $\epsilon > 0$, there exists $l_{\epsilon}$ such that the `area' of a minimal diagram of boundary length $l$ is at most $l^{2+\epsilon}$, for all $l \geq l_{\epsilon}$. \end{theorem}

As the isoperimetric function for any CAT(0) space is quadratic, this theorem applies to asymptotically CAT(0) spaces. One wonders if the above estimate can be improved to a quadratic bound. Nevertheless the isoperimetric function for any asymptotically CAT(0) group is sub-cubic. In \cite{SG}, Gersten demonstrates that a finitely presented group $G$ has solvable word problem if and only if the isoperimetric function for $G$ is recursive. We conclude from the above discussion that the word problem for asymptotically CAT(0) groups is solvable.

\subsubsection*{Nilpotent subgroups}

\begin{proposition} Every quasi-isometrically embedded nilpotent subgroup of an asymptotically CAT(0) group is virtually abelian. \end{proposition}

\begin{proof} Let $H$ be a quasi-isometrically embedded nilpotent subgroup of an asymptotically CAT(0) group $G$. Now, suppose that the group $G$ acts geometrically the space $X$. Then there exists a quasi-isometric embedding of $H$ into $X$. This means that there is a biLipschitz embedding of any asymptotic cone of $H$ into the corresponding asymptotic cone of $X$. 

We know from Gromov's work (see \cite{G}) on groups of polynomial growth that every asymptotic cone of $H$ is a graded Nilpotent Lie group with the left invariant Carnot metric. However, the main theorem of \cite{PS} shows that for asymptotically CAT(0) $X$, no such embedding is possible, unless $H$ is virtually abelian. \end{proof}

\section{Amalgams and HNN Extensions}\label{combination}

In this section, we provide methods for combining asymptotically CAT(0) groups via amalgams and HNN extensions.

\noindent \textbf{Amalgams and HNN Extensions with Isometric Gluing } We now describe techniques to form amalgams and HNN extensions from asymptotically CAT(0) groups. The hypotheses on the amalgamated subgroup in theorem \ref{amalgam} and \ref{hnn} appear restrictive at first. However the conditions are met among others, by finite subgroups with fixed points, virtually cyclic subgroups in CAT(0) groups or by the central infinite cyclic subgroups of co-compact lattices in $\slr$.

\begin{theorem} \label{amalgam} Let $G_1$, $G_2$ and $H$ be groups acting geometrically on asymptotically CAT(0) geodesic spaces $X_1$, $X_2$ and $A$ respectively. Suppose that for $i=1,2$, there exist monomorphisms, $\phi_i:H \rightarrow G_i$ and a $\phi_i$-equivariant isometric embedding $f_i: A \rightarrow X_i$. Then, the amalgam $G= G_1 *_H G_2$ associated to the maps $\phi_i$ acts geometrically on an asymptotically CAT(0) geodesic space. \end{theorem}

\begin{proof} The amalgam $G= G_1 *_H G_2$ acts simplicially on a tree $T$ which is unique up to graph isomorphism. The vertices of $T$ are in bijection with the cosets of $G_1$ and $G_2$ in $G$, while the unoriented egdes of $T$ may be identified with the cosets of $H$ in $G$. Given spaces on which the groups $G_1$, $G_2$ and $H$ act geometrically, one asks if there exists a space $Z$, which supports a geometric action of $G$ by isometries. Indeed, there is a well-known construction which serves this purpose, provided the maps $f_i$ and $\phi_i$ in the statement of the theorem exist. 

We present the construction in some detail here, following the treatment in Theorem II.11.18 of \cite{BH}. Start with an equivalence relation $\approx$ on the disjoint union of $G \times X_1$, $G\times [0,1] \times A$ and $G \times X_2$. The equivalence relation $\approx$ is generated by: $(gg_1,x_1)\approx (g, g_1x_1)$, $(gg_2,x_2)\approx (g, g_2x_2)$, $(gh,t, a)\approx (g,t, ha)$, $(g,f_1(a))\approx(g,0,a)$ and $(g,f_2(a))\approx(g,1,a)$ for all $g \in G$, $g_1 \in G_1$, $g_2 \in G_2$, $h \in H$, $x_1 \in X_1$, $x_2 \in X_2$, $a \in A$ and $t \in [0,1]$. 

For $i=1,2$, let $\bar{X}_i$ be the quotient of $G \times X_i$ by the above relation and similarly, let $\bar{A}$ be the quotient of $G\times [0,1] \times A$ by the above relation. Then $\bar{X}_i$ is isometric to $G/G_i \times X_i$; this is because each $g \times X_i$ contains exactly one element from each equivalence class of $G \times X_i$. Similarly, $\bar{A}$ is isometric to $G/H \times [0,1] \times A$. We are now in a position to describe $Z$. Recall that $T$ denotes the Bass Serre tree of $G$. 

The space $Z$ is a tree of spaces with underlying tree $T$ such that the vertex spaces are isomorphic to the $X_i$ and the edge spaces are isomorphic to $A$. More precisely,
\[Z := \frac{\left(G/G_1 \times X_1 \right)\coprod \left(G/H \times [0,1] \times A\right) \coprod\left(G/G_2 \times X_2\right)}{\sim}\]
The relation $\sim$ is given via the canonical surjections $G/H\rightarrow G/G_1$ and $G/H\rightarrow G/G_2$: $(gH,0,a) \sim (gG_1,f_1(a)) \textrm{ and } (gH,1,a)\sim (gG_2,f_2(a))$, for all $gH \in G/H$ and $a \in A$. 

The group $G$ acts by left multiplication on the first component of each of $G \times X_1$, $G\times [0,1] \times A$ and $G \times X_2$. This action is compatible with the gluing and so there is an induced action of $G$ on $Z$. The quotient of $Z$ via this action of $G$ is a compact space obtained via a gluing of $X_1/G_1$, $X_2/G_2$ along $A/H \times [0,1]$. Therefore the action is cocompact. 

If $G_1$, $G_2$ and $H$ act properly on $X_1$, $X_2$ and $A$, respectively, then $G$ acts properly on $Z$. The subgroup of $G$ leaving a copy of $M$, for $M\in \{X_1,X_2,A\}$, fixed is a conjugate of $G_i$  or of $H$. On the other hand, if an element $g$ of $G$ does not leave a copy of $M$ invariant then $g$ maps this copy of $M$ to a different one. Consequently, every point is moved by a distance of at least 2 and hence, the action of $G$ on $Z$ is proper. 

Endowed with the quotient metric, $Z$ is a geodesic space. There is a natural projection $\pi$ from $Z$ to the Bass Serre tree $T$ of $G$, which takes the equivalence classes of $(g,x_1)$, $(g,x_2)$ and $(g,t,a)$ to $(gG_1,0)$, $(gG_2,1)$ and $(gH,t)$, respectively. Moreover $\pi$ is $G$-equivariant. We describe a geodesic $\gamma$ joining the equivalence class $(gG_1,x_1)$ to the equivalence class of $(g'G_2,x_2)$. Recall that $A$ is a proper metric space and the maps $f_i$ are isometries. Therefore there is a point $\bar{x}_1$ in $(gG_1,f_1(A))$ which is closest to the point $(gG_1,x_1)$ in $Z$. Similarly there is a point $\bar{x}_2$ in $(g'G_2,f_2(A))$ which is closest to the point $(g'G_2,x_2)$ in $Z$. The geodesic $\gamma$ then is a concatenation of three geodesics, the first joining $(gG_1,x_1)$ to $\bar{x}_1$ in $(gG_1,X_1)$, the second joining $\bar{x}_1$ to $\bar{x}_2$ in $A \times T$ and the third, joining $\bar{x}_2$ to $(g'G_2,x_2)$ in $(g'G_2,X_2)$. We deduce from this and Bass Serre theory that the action of $G$ on $Z$ is by isometries. 

Our main task now is to show that $Z$ is asymptotically CAT(0). The space $Z$ supports a proper cocompact $G$-action and so, the choice of base point is not crucial. Let $\omega$ be a non-principal ultrafilter and choose a sequence $(a_n)$ of positive real numbers such that $\lim_{n\rightarrow \infty}a_n = \infty$. We will show that the canonical asymptotic cone (see Section \ref{intro}) $Z_\omega:= Cone_\omega(Z,(a_n))$ is a CAT(0) space.

Denote the canonical asymptotic cones of $X_1$, $X_2$, $T$ and $A$ with respect to the non-principal ultrafilter $\omega$ and scaling sequence $(a_n)$ by $(X_1)_\omega$, $(X_2)_\omega$, $T_\omega$ and $A_\omega$, respectively. Observe that the isometric embeddings $f_i$ induce isometries $F_i$ at the level of asymptotic cones. We construct a new space $\mathcal{Z}$ as a tree of spaces with underlying tree, $T_\omega$. 
\[\mathcal{Z} := \frac{\left(T_\omega \times (X_1)_\omega \right)\coprod \left(T_\omega \times A_\omega\right) \coprod\left(T_\omega \times (X_2)_\omega \right)}{\sim},\]
where, $(\underline{t},F_1(\underline{a})\sim (\underline{t},\underline{a}) \sim (\underline{t},F_2(\underline{a}))$, for all $\underline{a} \in A_\omega$ and $\underline{t} \in T_\omega$. 

The proof of the theorem will therefore follow from the next proposition. \end{proof} 

\begin{proposition} The spaces $Z_\omega$ and $\mathcal{Z}$ are isometric. Moreover, $\mathcal{Z}$ is CAT(0).
\end{proposition} 

\begin{proof} We first show that $\mathcal{Z}$ is CAT(0). We have assumed that the spaces $X_i$ and $A$ are asymptotically CAT(0). It follows that $T_\omega \times A_\omega$,  is CAT(0). Now applying Theorem II.11.3 of \cite{BH} multiple times, we see that $\mathcal{Z}$ is CAT(0). 

We now define a map $\eta$ from $Z_\omega$ to $\mathcal{Z}$, which furnishes us with the required isometry. Let $(z_n)\in Z_\omega$. Define $\mathcal{X}_1= \{n \in \mathbb{N}\ |\ z_n \in G/G_1 \times (X_1-f_1(A))\}$, $\mathcal{X}_2= \{n \in \mathbb{N}\ |\ z_n \in G/G_2 \times (X_2-f_2(A))\}$ and $\mathcal{A}= \{n \in \mathbb{N}\ |\ z_n \in G/H \times [0,1] \times A\}$. Then, $\mathcal{X}_1 \coprod \mathcal{A} \coprod \mathcal{X}_2 = \mathbb{N}$. This implies that exactly one of these three sets has $\omega$-measure 1, and so $z_n$ belongs to exactly one of $T \times A$, $G/G_1 \times X_1$ and $ G/G_2 \times X_2$ with $\omega$-measure 1.

Observe that the copy of $(T\times A)_\omega \cong T_\omega \times A_\omega$ in $Z_\omega$ is isometric to the copy of $T_\omega \times A_\omega$ in $\mathcal{Z}$. Therefore, the restriction of $\eta$ to $T_\omega \times A_\omega \subset Z_\omega$ can be taken to be the identity map. 

Now suppose that $\omega(\mathcal{X}_1)=1$. For each $n \in \mathbb{N}$, let $t_n$ be the projection of $z_n$ on to the copy of $T$ in $Z$. Since the projection map decreases distances, $(t_n)$ defines a point in the tree $T_\omega$ in $Z_\omega$. For $n \in \mathcal{X}_1$, define $w_n$ to be the projection of $z_n$ onto $X_1$, otherwise take $w_n$ to be any point in $X_1$. Let $\eta((z_n))=((t_n),(w_n))$. Similarly, define $\eta$ for the case when $\omega(\mathcal{X}_2)=1$. 

Observe that the copies of $(T \times A)_\omega$, $(X_1)_\omega$ and $(X_2)_\omega$ in $Z_\omega$ and $\mathcal{Z}$ are isometric and moreover $\eta$ is a bijection. We deduce that $\eta$ defines an isometry from $Z_\omega$ on to $\mathcal{Z}$. This proves that $Z_\omega$ is CAT(0). \end{proof}

Along the same lines, one can construct HNN extensions of asymptotically CAT(0) groups.  

\begin{theorem}\label{hnn} Let $G$ and $H$ be groups acting properly by isometries on asymptotically CAT(0) spaces $X$ and $A$. Suppose that for $i=1,2$, there exist monomorphisms $\phi_i: H \rightarrow G$ and $\phi_i$-equivariant embedding $f_i:Y \rightarrow X$. Then the HNN extension $G*_H$ acts properly by isometries on an asymptotically CAT(0) space.\end{theorem} 

\begin{proof}The proof is similar to that of the previous theorem. The only difference lies in the construction of the space $Z$. For an HNN extension $\Gamma$ of $G$ over the subgroup $H$, define $Z$ to be as follows.
\[Z := \frac{\left(\Gamma/G_1 \times \tilde{X} \right)\coprod \left(\Gamma/H \times [0,1] \times A\right)}{\sim} \]
where $(\gamma H, 0, a)\sim (\gamma G, f_1(a)) \textrm{\ and\ } (\gamma H,1,a)\sim (\gamma G,f_2(a))$, for all $a\in A$ and for all $\gamma \in \Gamma$. Conclude as before that $Z$ and hence $\Gamma$ is asymptotically CAT(0). \end{proof}

\noindent \textbf{Amalgams and HNN Extensions along Finite Subgroups } 

\begin{theorem}\label{amalgamfinite} Let $G_1$ and $G_2$ be asymptotically CAT(0) groups and let $C$ be a finite group, endowed with monomorphisms $\phi:C\rightarrow G_1$ and $\psi:C \rightarrow G_2$. Then the amalgam $G:= G_1*_CG_2$ associated to $\phi$ and $\psi$ is also asymptotically CAT(0). Similarly, the class of asymptotically CAT(0) groups is closed under HNN extensions along finite subgroups. \end{theorem}

\begin{proof} Let $G_1$, $G_2$ and $C$ be as above; let $X_1$ and $X_2$ be the asymptotically CAT(0) spaces associated to $G_1$ and $G_2$ respectively. To prove the theorem, we need to construct an asymptotically CAT(0) space which supports a geometric $G$-action. 

Fix $x_1 \in X_1$ and $x_2 \in X_2$. Let $H_1$ be the stabilizer of $x_1$ in $G_1$. Likewise, denote the stabilizer of $x_2$ in $G_2$ by $H_2$. For fixed $x_1 \in X_1$, the set map $G_1 \rightarrow G_1/H_1$ gives a canonical map $\pi_1: G_1 \rightarrow G_1.x_1$ from the group $G_1$ to the orbit $G_1.x_1$. Similarly, there exists a natural map $\pi_2: G_2 \rightarrow G_2.x_2$. Define  
\[Z := \frac{\left(G/G_1 \times X_1 \right)\coprod \left(G/C \times [0,1] \times C\right) \coprod\left(G/G_2 \times X_2\right)}{\sim}\]
where $(\gamma C,0,c)\sim (\gamma G_1, \pi_1 \circ \phi(c)) \textrm{\ and\ } (\gamma C,1,c)\sim (\gamma G_2, \pi_2 \circ \psi(c))$, for all $c\in C$ and for all $\gamma \in G$. 

By a similar argument as before, the amalgam $G$ acts properly and co-compactly on $Z$ by isometries. We claim that $Z$ is asymptotically CAT(0). There is a natural projection of $Z$ to the Bass Serre tree $T$ of $G$. Consider the tree of spaces $\tilde{Z}$ with underlying tree $T$ and vertex spaces $X_1$ and $X_2$; that is, a vertex of the form $gG_1$ of $T$ corresponds to a copy of $X_1$ and a vertex of the form $gG_2$ of $T$ corresponds to a copy of $X_2$. Observe that $Z$ is $(1,\epsilon)$-quasi-isometric to $\tilde{Z}$, where $\epsilon$ depends solely on the diameter of $C.x_1$ in $X_1$ and the diameter of $C.x_2$ in $X_2$. 

It follows from the proof of Theorem \ref{amalgam} that $\tilde{Z}$ is asymptotically CAT(0). Moreover, a $(1,\epsilon)$ quasi-isometry induces an isometry at the level of asymptotic cones. Hence, the space $Z$ is asymptotically CAT(0). \end{proof}

\section{Relative Hyperbolicity}\label{relative}
The aim of this section is to prove the following theorem: 

\begin{theorem} \label{RelH} If a group $G$ is relatively hyperbolic with respect to a subgroup $H$ and $H$ is asymptotically CAT(0), then $G$ is asymptotically CAT(0). \end{theorem} 

Relatively hyperbolic groups were first introduced by Gromov in \cite{MG} and later studied by Farb in \cite{farb} and Bowditch in \cite{bow} among others. The motivating examples were the fundamental groups of complex hyperbolic manifolds with cusps. The presence of the cusp subgroups ensure that these groups are not negatively  curved. We refer the reader to \cite{DS} for an introduction to relative hyperbolicity. 

\begin{proof} [Proof of Theorem \ref{RelH}] Let a group $G$ be hyperbolic relative to an asymptotically CAT(0) subgroup $H$. There exists a geodesic space $X$ such that all asymptotic cones of $X$ are CAT(0) and $H$ acts geometrically on $X$. In \cite{thesis} and \cite{DG} one can find a description of a geodesic metric space $Y$ which supports a geometric $G$-action. We need to show that all asymptotic cones of $Y$ are CAT(0). 
Let $\omega$ be a non-principal ultrafilter and $(a_n)$, a sequence of positive real numbers such that $\lim_{n\rightarrow \infty}a_n = 0$. Choose for each $n \in \mathbb{N}$, a point $p_n \in Y$. Set $\mathcal{X}=\left\{Cone_{\omega}(Z)\ |\ Z \in \Xi\right\}$, where $Cone_{\omega}(Z)$ refers to the $\omega$-limit of $Z$ in $Cone_{\omega}(Y, (a_n), (p_n))$ and $\Xi$ is in bijection with the cosets of $H$ in $G$.

Clearly, $Y_{\omega}$ is a complete geodesic space. As the action of $G$ on $Y$ is proper and co-compact, $G$ is quasi-isometric to $Y$. By Theorem 1.11 of \cite{DS}, $G$ is asymptotically tree graded with respect to the cosets of $H$. Moreover, Theorem 5.1 in \cite{DS} states that the property of being asymptotically tree-graded is  preserved under quasi-isometries. It follows, that $Y$ is asymptotically tree-graded with respect to $\mathcal{X}$. 

The asymptotic cone $Y_{\omega}$ is tree-graded with respect to $\mathcal{X}$ and by hypothesis, each piece is CAT(0). We know that every simple triangle in $Y_{\omega}$ is contained in a piece. Hence, we may assume that our triangle $ABC$ in $Y_{\omega}$ has the form $A'B'C' \cup AA'\cup BB' \cup CC'$, where $A'B'C'$ is a geodesic triangle that lies in some piece of the asymptotic cone while $AA'$, $BB'$ and $CC'$ are simply geodesics. 

We will use the `Bruhat-Tits' inequality for CAT(0) spaces. This says that a geodesic space $\mathcal{M}$ is CAT(0) if and only if for all triples $(p,q,r) \in \mathcal{M}^3$ and all $m \in \mathcal{M}$ with $d(q,m)=d(m,p)=d(p,q)/2$, we have $d(p,q)^2+d(q,r)^2$ $\geq 2d(m,p)^2+ d(q,r)^2/2$.

To show that the triangle $ABC$ satisfies the CAT(0) property, take $M$ to be the midpoint of the side $BC$. The case when $M$ lies on the geodesic $BB'$ or the geodesic $CC'$ is trivial. So assume that $M \in B'C'$. There is a comparison triangle $A_1'B_1'C_1'$ with comparison point $M_1$ on $B_1'C_1'$ for $M$. Since each piece of the asymptotic cone is CAT(0) we have $d(A', M)\leq d(A_1',M_1)$.

Let $\bar{A}\bar{B}\bar{C}$ be a comparison triangle for $ABC$ with comparison point $\bar{M}$ for $M$. Let $a=d(A',B')$, $b=d(A',C')$, $c=d(B',C')$, $x=d(M,B')$, $r=d(A,A')$, $p=d(B,B')$, $q=d(C,C')$, $h=d(\bar{A},\bar{M})$ and $h'=d(A_1',M_1)$. We know that $d(A,M) = $ $d(A,A') + d(A',M)\leq h'+r$. Hence it suffices to prove that $h'+r \leq h$.

\noindent \textbf{Case 1. The value of $r$ is 0. } Note that $x=(c+q-p)/2$, so using the Cosine Law,
\[ h'^2= \frac{2a^2+2b^2-c^2}{4} + \frac{(q-p)^2}{4} + \left(\frac{q-p}{2}\right)\left(\frac{b^2-a^2}{c}\right).\quad \quad \quad \quad \quad\quad \quad \quad \quad \quad\]
On the other hand, \[h^2 = \frac{2a^2+2b^2-c^2}{4}+\frac{(q-p)^2}{4}+\frac{4ap+4bq-2pc-2cq}{4}.\quad \quad \quad \quad \quad \quad \quad \quad \quad \]
\[\textrm{Hence, } h^2-h'^2 = \frac{p (b^2-(a-c)^2)+ q(a^2-(b-c)^2)}{2}.\quad \quad \quad \quad \quad \quad\quad  \quad \quad \quad\quad\]That the final expression is non-negative is a consequence of the triangle inequality for $A'B'C'$.

\noindent \textbf{Case 2. The value of $r$ is not zero. } By case 1, the result holds for the triangle $A'BC$. Now, let $\alpha = d(A',B)$, $\beta=d(A', C)$ and $\gamma=d(B,C)$. Further, set $\alpha' = \alpha + r$ and $\beta'=\beta + r$.

Then, \[h'+r= \sqrt{\frac{2\alpha^2 + 2\beta^2-\gamma^2}{4}}+r, \textrm{ and\ \ }h= \sqrt{\frac{2\alpha'^2 +2\beta'^2-\gamma^2}{4}}.\]
Manipulating the above two expressions, one reduces the inequality $h'+r \leq h$ to $\sqrt{2\alpha^2 + 2\beta^2-\gamma^2}\leq\alpha + \beta$ or equivalently to $(\alpha+\beta)^2 - \gamma^2 \leq 0$. This again is a consequence of the triangle inequality. This proves that $Y_{\omega}$ is CAT(0). \end{proof}

Theorem \ref{RelH} and Theorem B in \cite{TE} together imply the corollary below. 

\begin{corollary}\label{systolic} Let $X$ be a systolic complex with the Isolated Flats Property and $G$, a group acting cocompactly and properly discontinuously on $X$. Then the group $G$ is asymptotically CAT(0). \end{corollary}

\begin{question} Are all systolic groups, asymptotically CAT(0)? \end{question}

\section{The Universal Cover of $SL(2,\mathbb{R})$}\label{PSL}

Virtually, co-compact lattices in $\slr$ are fundamental groups of $T^1(S)$, where $T^1(S)$ denotes the unit tangent bundle of a closed surface $S$ of genus at least 2. Co-compact lattices in $\slr$ are not hyperbolic because they contain free abelian subgroups of rank 2. Moreover they cannot act properly by semisimple isometries on any CAT(0) space. This is explained in the proof of Theorem II.7.26 of \cite{BH}.

In this section we consider $\slr$, endowed with the Sasaki metric, which is a left invariant Riemannian metric and show that it is asymptotically CAT(0). It is well known that $PSL(2,\mathbb{R})$ acts on the hyperbolic plane by isometries. This action can be used to identify $PSL(2,\mathbb{R})$ with the unit tangent bundle $T^1(\mathbb{H}^2)$ of the hyperbolic plane. The universal cover of the latter is then $\slr$.

The tangent bundle $TM$ of a Riemannian manifold $M$ can be given a Riemannian metric. There is therefore an induced Riemmanian metric on the unit tangent bundle $T^1(M)$ of $M$. Unit tangent bundles of Riemannian manifolds were studied in some detail by Sasaki in \cite{S1} and \cite{S2}. 

\noindent \textbf{Convention } The word `geodesic' will refer to curves in a Riemannian manifold with constant speed parametrization in this section alone. More precisely, if $M$ is a Riemannian manifold and $\nabla$ is its Riemannian connection, then a geodesic in $M$ is a curve $\gamma$ such that $\nabla_{\frac{d\gamma}{dt}}\frac{d\gamma}{dt}=0$.

\subsection{The Sasaki Metric on $T(M)$}

Let $M$ be an $n$-dimensional Riemannian manifold and let $TM$ be its tangent bundle. Consider the projection $\pi :TM \rightarrow M$ that sends a point $\theta = (x,v)$ in $TM$ to the point $x \in M$. The map $\pi$ is a Riemannian submersion. The kernel $V(\theta)$ of its differential is made of vectors in the fibre $T_{\theta}TM$ at $\theta$, that are tangent to the fibre $T_x(M)$ at $x$. These vectors are said to be `vertical'. On the other hand, the vectors which are orthogonal to the fibre at $x$ are the `horizontal' vectors. These are denoted $H(\theta)$ and more formally, they form the kernel of the covariant map.

A curve $\sigma$ in $TM$ is given by the pair $(\alpha(t), v(t))$, where $\alpha$ is a path in $M$ and $v(t)$ is a vector field along $\alpha$. If $V$ is an element of $T_{\theta}TM$, then $V$ comes from an infinitesimal path $\sigma: (-\epsilon, \epsilon) \rightarrow TM$, which satisfies $\sigma'(0)=V$. One may now compute the covariant derivative of the vector field $v(t)$ along $\alpha'$. This measures the rate at which $v(t)$ varies from the tangent vector to the curve $\alpha$.

The covariant map $K_{\theta}$ at the point $V$ is defined to be $(\nabla_{\alpha'}v)(0)$, where $\nabla$ denotes the Riemannian connection of $M$. One can show that $H(\theta)$ is precisely the kernel of the covariant map, that the linear map $d_{\theta}\pi$ gives an isomorphism of $T_xM$ with $H(\theta)$ while $K_{\theta}$ gives a linear isomorphism of $V(\theta)$ with $T_xM$. Moreover, the vector space $T_{\theta}TM$ is a direct sum of $H(\theta)$ and $V(\theta)$.

One defines the Sasaki metric $\left\langle \left\langle .,.\right\rangle\right\rangle_{\theta}$ on $T_{\theta}TM$ so that these two components are orthogonal. For $V$ and $W$ in $T_{\theta}TM$, 
\[\left\langle \left\langle V,W \right\rangle\right\rangle_{\theta}:=\left\langle d_{\theta}\pi(V), d_{\theta}\pi(V)\right\rangle_x + \left\langle K_{\theta}(V),K_{\theta}(V)\right\rangle_x.\]

Note that a curve $\sigma (t)=(\alpha(t),v(t))$ in  $TM$ is \textit{horizontal} if its tangent vector is horizontal, which is the same as saying that the vector field $v(t)$ is the parallel transport of its initial vector along $\alpha(t)$.

The collection of unit tangent vectors $T^1(M)$ is a Riemannian subspace of $TM$ with the induced Riemannian metric. 

\subsection{Geodesics in $T^1(\mathbb{H}^2)$}

We will now describe the geodesics of the unit tangent bundle of the hyperbolic plane, endowed with the induced Sasaki metric. The account in this paragraph follows Sasaki's work from \cite{S2}.

A curve $\Gamma$ on $T^1(\mathbb{H}^2)$ is a unit vector field $y(\sigma)$ along a curve $x(\sigma)=\pi(\Gamma)$ in $\mathbb{H}^2$, where $\sigma$ is the arc length of $\Gamma$. Let $x'$ denote $\frac{dx}{d\sigma}$ and $\nabla$ denote the Riemannian connection of $\mathbb{H}^2$. (We work throughout with the upper half plane model of the hyperbolic plane). Then $\left\langle \left\langle \Gamma',\Gamma'\right\rangle\right\rangle =1$, which is equivalent to 
\[\left\langle x',x'\right\rangle+\left\langle \nabla_{x'}y, \nabla_{x'}y\right\rangle = 1.\]

Putting $c^2=\left\langle \nabla_{x'}y, \nabla_{x'}y\right\rangle$ we have $\left\langle x',x'\right\rangle=1-c^2$ and $0\leq c \leq1$. The conditions for $\Gamma$ to be a geodesic in $T^1(\mathbb{H}^2)$ are that $c$ is a constant and $x(\sigma)$ and $y(\sigma)$ satisfy the differential equations 
\[x''=by-a \nabla_{x'}y, \ \nabla_{x'}\nabla_{x'}y=-c^2y.\]

where $a=\left\langle x',y\right\rangle$ and $b=\left\langle x',\nabla_{x'}y\right\rangle$.

Using $c$, one may characterize the geodesics in $T^1(\mathbb{H}^2)$ into the following types: 

\begin{enumerate}

\item Horizontal type or $c=0$: In this case, $ \nabla_{x'}y=0$ and so $\Gamma$ is a horizontal geodesic in the unit tangent bundle. Its corresponding $\pi$ image in the hyperbolic plane is also a geodesic. 

\item Vertical type or $c=1$: In this case the image of $\Gamma$ under $\pi$ is a point and the geodesic $\Gamma$ is a great circle that lives entirely in the fibre above that point. 

\item Oblique type or $0<c<1$ : Let $T$, $N$ denote the unit tangent vector and the principal normal vector of $x$ in $\mathbb{H}^2$; let $\kappa$ be the curvature of $x$ and $s$, its arclength. Then, $\frac{ds}{d\sigma}= \sqrt{1-c^2}$ and $x'=\sqrt{1-c^2}T$ while $x''=(1-c^2)\kappa N$. One can show that $(1-c^2)^2\kappa^2=c^2$ and so $\kappa$ is always constant. Hence, $x$ is an equidistant curve, a horocycle or a circle in the hyperbolic plane, depending on whether $\kappa^2$ is less than 1, equal to 1 or greater than 1.

\end{enumerate}

\subsection{Asymptotic Cones of $\slr$}

\begin{theorem} \label{slr} There exists a $(1,\pi)$-quasi-isometry from $\slr$, endowed with the Sasaki metric to the CAT(0) space $\mathbb{H}^2\times \mathbb{R}$.\end{theorem}

\begin{proof} The proof exploits the structure of $\slr$ as a Riemannian manifold. We saw that the action of $PSL(2,\mathbb{R})$ on the hyperbolic plane gives an identification of $PSL(2,\mathbb{R})$ with the unit tangent bundle $T^1(\mathbb{H}^2)$ of $\mathbb{H}^2$. Thus, $\slr$ is the universal cover of $T^1(\mathbb{H}^2)$. One uses the Sasaki metric to make the tangent bundle of $\mathbb{H}^2$ into a Riemannian manifold. The Riemannian metric on $T^1(\mathbb{H}^2)$ is the induced Sasaki metric. Let $\pi: \slr \rightarrow \mathbb{H}^2$ denote the canonical projection of $\slr$ onto the hyperbolic plane. 

We now describe an identification of $\slr$ with $\mathbb{H}^2 \times \mathbb{R}$. Fix a base point $*$ on $\mathbb{H}^2$ and a reference unit vector $v \in T_*(\mathbb{H}^2)$. For any curve $\alpha \in \mathbb{H}^2$, denote by $P_{\alpha}(w)$, the parallel transport of a vector $w$ along $\alpha$. Using $(*,v)$, form a section $s$ of $T^1(\mathbb{H}^2) \rightarrow \mathbb{H}^2$ as follows: given $x\in \mathbb{H}^2$, let $\gamma$ be the unique geodesic joining $*$ to $x$ in the hyperbolic plane. Define $s(x)=(x, P_{\gamma}(v))$. Since $\mathbb{H}^2$ is simply connected, the section $s$ lifts to a section $\tilde{s}$ of $\mathbb{H}^2$ to $\slr$. With this global section, one can describe a point $P$ in $\slr$ with an ordered pair $(\pi(P), \theta(P))\in \hr$, where $\pi(P)$ is the projection of $P$ to $\mathbb{H}^2$ and $\theta(P)$ is the distance of $P$ from $\tilde{s}(\pi(P))$ in $\slr$. 

We know that a geodesic in $\slr$ projects, via the map $\pi$, to one of the following: a point, an $\mathbb{H}^2$-geodesic or an arc of a proper circle, a horocycle or an equidistant curve. Since a $\slr$-geodesic $\eta$ is simply a vector field along the projection $\bar{\eta}:=\pi(\eta)$, the general form of a geodesic joining two points, $P$ and $Q$ (which are identified with $(\pi(P),\theta(P))$ and $(\pi(Q),\theta(Q))$ respectively) is given by $\eta(t)=(\bar{\eta}(t),P_{\bar{\eta}}(w)+tc_{\eta} )$, where $w$ is the initial vector of $\eta$ and $|c_{\eta}|\leq 1$ is the rate of rotation of $w$ along $\bar{\eta}$. 

Let $\eta$ be a distance-minimising geodesic in $\slr$ from $P$ to $Q$. We assume for simplicity that $\theta(Q) \geq \theta(P)$. The length $l(\eta)$ in the Sasaki metric is given by $l(\eta)=\sqrt{l(\bar{\eta})^2+(\theta(Q)-\theta(P)-P_{\bar{\eta}}(w))^2}$. 

Observe that if the points $P$ and $Q$ are joined by the curve $\alpha$ in $\slr$ whose projection in $\mathbb{H}^2$ is the geodesic joining $\pi(P)$ and $\pi(Q)$, then $l(\eta) \leq l(\alpha)$, because $l(\alpha)=\sqrt{d(\pi(P),\pi(Q))^2+(\theta(Q)-\theta(P))^2}$. On the other hand, $|P_{\bar{\eta}}(w)|\leq \pi$, which implies that if $\theta(Q)-\theta(P) \geq \pi$ then the length of $\eta$ is at least $\sqrt{l(\bar{\eta})^2+(\theta(Q)-\theta(P)-\pi)^2}$. 

Let $d=d(\pi(P),\pi(Q))$, $r=\theta(Q)-\theta(P)$, $L=l(\eta)$ and $D=\sqrt{d^2+r^2}$. Note that $D$ is the distance between the images of $P$ and $Q$ in $\hr$, while $L$ is the distance between them in $\slr$. As the length $L=l(\eta)$ of $\eta$ is no larger than $\sqrt{d(\pi(P),\pi(Q))^2+(\theta(Q)-\theta(P))^2}$, we deduce that $L \leq D$. 

If $\theta(Q)-\theta(P) \leq \pi$ then by the triangle inequality, $D\leq d + r$ $\leq L+\pi$. If however, $\theta(Q)-\theta(P) \geq \pi$, then $l(\eta) \geq \sqrt{d(\pi(P),\pi(Q))^2+(\theta(Q)-\theta(P)-\pi)^2}$. So, $d^2 + (r-\pi)^2 \leq L^2$. Since $r\leq L+\pi$, $D^2 \leq (L+\pi)^2$.

In all cases, we have $L \leq D \leq L + \pi$. \end{proof}

\begin{corollary}The Lie group $\slr$ is asymptotically CAT(0).\end{corollary}

\begin{proof}By Corollary \ref{isometry}, a $(1,\pi)$-quasi-isometry induces an isometry at the level of asymptotic cones and so every asymptotic cone of $\slr$ is a direct product of the real line with the infinitely branching homogeneous $\mathbb{R}$-tree.\end{proof}

\section{Interesting Questions}

There are many interesting questions that one may formulate about asymptotically CAT(0) groups. Here is a selection. 

\noindent \textbf{Asymptotically CAT(0) graphs. } The conjecture below is commonly attributed to Erd\'os and a mention of it maybe found in \cite{pach}.
\begin{conjecture} The integer points in the Euclidean Plane may be connected to form a graph which is $(1,k)$-quasi-isometric to the Euclidean plane. \end{conjecture}

A graph is $\delta$-CAT(0) if and only if it is hyperbolic (see \cite{thesis}). One wonders if this is also the case with asymptotically CAT(0) graphs. I would like to propose the following conjecture.

\begin{conjecture}A graph is asymptotically CAT(0) if and only if it is $\delta$-hyperbolic. \end{conjecture}

Observe that an affirmative answer to the second conjecture implies a negative answer to the first. Indeed, if there exists a graph $X$ with vertex set $\mathbb{Z}^2$ such that $X$ is $(1,k)$-quasi-isometric to the Euclidean plane, then all asymptotic cones of $X$ are isometric to the Euclidean plane and thus CAT(0). But then, $X$ is hyperbolic, which contradicts the assumption that $X$ is quasi-isometric to the Euclidean plane. 

\noindent \textbf{Novikov's Conjecture for Asymptotically CAT(0) Groups. } There are many different approaches by which the Novikov Conjecture may be proved for asymptotically CAT(0) groups. One would be to develop an asymptotic notion of $\delta$-bolicity so that existing techniques from \cite{bolic} may be extended to establish the conjecture.


\begin{thebibliography}{999}
\bibitem{bow} B. Bowditch, Relatively hyperbolic groups, preprint.
\bibitem{BH} M. R. Bridson and A. Haefliger, Metric Spaces of Non-positive Curvature, Springer, Berlin 1999.
\bibitem{delzant} T. Delzant and M. Gromov, Courbure mesoscopic et theorie de la toute petite simplification, J. Topol. 1 (2008), no. 4, 804-836.
\bibitem{D} C. Drutu, Quasi-isometry Invariants and Asymptotic Cones, Intern. J. of Algebra and Comp.,12 (2002) no.1-2, 99-135
\bibitem{DS} C. Drutu and M. Sapir, Tree graded spaces and Asymptotic Cones of Groups, Topology \textbf{44},(2005) 959-1058
\bibitem{TE} Tomasz Elsner, Systolic Groups with Isolated Flats, unpublished.
\bibitem{farb} B. Farb, Relatively hyperbolic groups, GAFA 8 (1998), 810–840.
\bibitem{SG} S.M.Gersten, Isoperimetric and Isodiametric Functions of Finite Presentations, Geometric Group Theory (vol 1), G.A.Niblo, M.A.Roller (eds), Proceedings of the Symposium held in Sussex, LMS Lecture Notes Series 181, CUP 1991. 
\bibitem{MG} M. Gromov, Hyperbolic groups, Essays in Group Theory, (S.M. Gersten, ed), Springer Verlag, MSRI Publ. \textbf{8} (1987), 75-263.
\bibitem{G} M. Gromov, Groups of Polynomial growth and Expanding Maps, Publications mathematiques I.H.É.S., 53, 1981.
\bibitem{G1} M. Gromov, Mesoscopic curvature and hyperbolicity. Global differential geometry: the mathematical legacy of Alfred Gray (Bilbao 2000), 58-69, Contemp. Math., 288, Amer. Math. Soc., 2001.
\bibitem{DG} D. Groves, Limit Groups for Relatively Hyperbolic Groups I: The Basic Tools, arXiv:math/0412492v2 
\bibitem{JS} T. Januszkiewicz, J. Swiatkowski, Simplicial non-positive curvature, Publ.Math. IHES, 104(2006), 1-85.
\bibitem{thesis} A. Kar, Discrete Groups and CAT(0) Asymptotic Cones, Doctoral Dissertation, The Ohio State University 2009.
\bibitem{bolic} G. Kasparov and K. Skandalis, Groups acting properly on bolic spaces and the Novikov Conjecture, Ann. of Math. (2) 158 (2003), 1, 165-206.
\bibitem{OOS} A. Olshanskii, D. Osin and M. Sapir, Lacunary Hyperbolic Groups, arXiv:math/0701365
\bibitem{pach} J. Pach, R. Pollack and J. Spencer, Graph distance and Euclidean Distance on the Grid, Topics in Combinatorics and Graph Theory, R. Bodendieck, R. Henn (Eds.), Physica-Verlag Heidelberg 1990.
\bibitem{PS} Scott Pauls, The large scale geometry of nilpotent Lie Groups, Comm. Anal. Geom. 5(5) pp 951-982, 2001.
\bibitem[13]{TR} T. Riley, Asymptotic Invariants of Infinite Discrete Groups, Thesis, University of Oxford.
\bibitem[14]{S1} Shigeo Sasaki, On the Differential Geometry of Tangent Bundles of Riemannian Manifolds I, II, Tohoku Math. J.\textbf{10}(1958), 338-354, \textbf{14}(1962),146-155
\bibitem[15]{S2} Shigeo Sasaki, Geodesics on the Tangent Sphere Bundles over Space Forms, J. Reine Agnew. 
Math. \textbf{288} (1976), 106-120
\bibitem[16]{J-P.S} J.P. Serre, Arbres, Amalgams, $SL_2$; Asterique 46, 1977.
\bibitem[17]{TV} S. Thomas and B. Velickovic, Asymptotic cones of finitely generated groups. Bull. London Math. Soc. 32 (2000), no. 2, 203-208.
\end{thebibliography}
\end{document}